\documentclass[12pt,reqno]{amsart}

\advance \topmargin by -\headheight
\advance \topmargin by -\headsep
\evensidemargin \oddsidemargin
\usepackage{amsmath}
\usepackage{amsfonts}
\usepackage{stmaryrd}
\usepackage{amssymb}
\usepackage{amsthm}
\usepackage{csquotes}

\usepackage{mathrsfs}
\usepackage{dsfont}
\usepackage{bm}
\usepackage{cite}

\numberwithin{equation}{section}

\newcommand\ls{\lesssim}
\newcommand\gs{\gtrsim}

\newtheorem{theorem}{Theorem}[section]

\newtheorem{lemma}[theorem]{Lemma}
\newtheorem{Proposition}[theorem]{Proposition}

\usepackage{mathtools}
\DeclarePairedDelimiter{\ceil}{\lceil}{\rceil}

\title[Sum-Product estimates for diagonal matrices]{Sum-Product estimates for diagonal matrices}
\author[Akshat Mudgal]{Akshat Mudgal}
\date{} 
\address{Department of Mathematics, Purdue University, 150 N. University Street, West Lafayette, IN 47907-2067, USA}
\email{am16393@bristol.ac.uk, amudgal@purdue.edu}

\begin{document}

\maketitle
\begin{abstract}
Given $d \in \mathbb{N}$, we establish sum-product estimates for finite, non-empty subsets of $\mathbb{R}^d$. This is equivalent to a sum-product result for sets of diagonal matrices. In particular, let $A$ be a finite, non-empty set of $d \times d$ diagonal matrices with real entries. Then for all $\delta_1 < 1/3 + 5/5277$, we have
\[  |A+A| + |A\cdot A| \gg_{d} |A|^{1 + \delta_{1}/d}. \]
In this setting, the above estimate quantitatively strengthens a result of Chang.
\end{abstract}

\section{Introduction}

Let $d$ be a natural number, and let $a = (a_1, \dots, a_d)$ and $b = (b_1, \dots, b_d)$ be elements of $\mathbb{R}^d$. We can define the sum and product of $a$ and $b$ as
\[ a + b = (a_1 + b_1, \dots, a_d + b_d) \ \text{and} \ a \cdot b = (a_1 b_1, \dots, a_d b_d) \]
respectively. In general let $R$ be a ring and let $A,B$ be finite, non-empty subsets of $R$. We can then define the sumset and the product set of $A$ and $B$ as 
\[ A+ B = \{ a + b \ | \ a \in A, b \in B \} \ \text{and} \ A \cdot B = \{ a \cdot b \ | \ a \in A, b \in B \} \]
respectively. For our purposes, $R$ will either be $\mathbb{R}^d$, or the set of $d \times d$ matrices with real entries, with $d$ being some natural number. 
\par

For ease of exposition, we will use Vinogradov notation, that is, we will write $X \gg Y$, or equivalently $Y \ll X$, to mean $|X| \geq C |Y|$ where $C$ is some positive constant. Moreover, we write $X \gs Y$, or equivalently $Y \ls X$, to mean $|X| \geq C|Y| (\log{|A|})^{D}$ where $C$ and $D$ are constants, and $C>0$. The corresponding notation $X \gg_{z} Y$ and $X \gs_{z} Y$ are defined similarly, except in this case the constants $C$ and $D$ will depend on the parameter $z$.
\par

When $d=1$, the sum-product conjecture states that for all finite, non-empty sets $A \subseteq \mathbb{R}$, we have
\begin{equation} \label{sumprodreal}
 |A+A| + |A\cdot A| \gs |A|^{1 + \delta_{1}},
 \end{equation}
for all $\delta_{1} \in (0,1)$. We will use $\delta_1$ to denote the best constant for which we know $\eqref{sumprodreal}$ to hold. The current record in this direction rests with Shakan \cite{Sh2018} who showed that $\delta_{1} = 1/3 + 5/5277$ is permissible\footnote{We note a recent improvement in this direction by Rudnev and Stevens who show that $\delta_1 <1/3 +  2/1167$ is permissible in $\eqref{sumprodreal}$, see preprint arXiv:2005.11145.}. Our main result extends the sum-product phenomenon to sets in $\mathbb{R}^d$ for $d >1$. 

\begin{theorem} \label{main}
Let $d \in \mathbb{N}$ and let $A \subseteq \mathbb{R}^d$ be a finite, non-empty set. Then we have
\[ |A+A| + |A\cdot A| \gtrsim_{d} |A|^{1 + \delta_1 /d}.\]

\end{theorem}
\par

We note that in $\mathbb{R}^2$, Theorem $\ref{main}$ is conjecturally optimal. In particular, if we assume that the sum-product conjecture holds, then we can set $\delta_1 = 1 - \varepsilon$, for any $\varepsilon > 0$. Combining this with Theorem $\ref{main}$, we find that for all $\varepsilon > 0$, and finite, non-empty sets $A \subseteq \mathbb{R}^2$, we have
\begin{equation} \label{conjineqeq}
  |A+A| + |A\cdot A| \gs |A|^{1 + 1/2 - \varepsilon/2} = |A|^{3/2 - \varepsilon/2}.
\end{equation}
This is optimal up to factors of $|A|^{\varepsilon}$, as we see with the following example. Let $A_N, B_N$ be sets of real numbers defined as
\[ A_N = \{1, 2, \dots, N\} \ \text{and} \ B_N = \{ 2, 4, \dots, 2^N \}. \]
Moreover, we let $C_N = A_N \times B_N \subseteq \mathbb{R}^2$. 
We note that 
\[ |C_N + C_N| = |A_N + A_N| |B_N + B_N| \ll N N^2 = N^{3} = |C_N|^{3/2}, \]
and similarly, 
\[ |C_N \cdot C_N| = |A_N \cdot A_N| |B_N \cdot B_N| \ll N^{2} N = |C_N|^{3/2}.\]
Thus we have
\begin{equation} \label{limit}
 |C_N + C_N| + |C_N \cdot C_N| \ll |C_N|^{1 + 1/2},
 \end{equation}
which matches the conjectural lower bound $\eqref{conjineqeq}$ up to factors of $|A|^{\varepsilon}$.
\par

We observe that by our definition of multiplication in $\mathbb{R}^d$, for every subset $A \subseteq \mathbb{R}^d$, there is a corresponding set $B$ of $d \times d$ diagonal matrices with real entries, such that $|A+A| = |B+B|$ and $|A\cdot A| = |B\cdot B|$, and vice versa. Thus Theorem $\ref{main}$ is equivalent to the following result. 

\begin{theorem} \label{sumprodd}
For all finite, non-empty sets $A$ of $d \times d$ diagonal matrices with real entries, we have
\begin{equation} \label{sumproddiag}
 |A+A| + |A\cdot A| \gs_{d} |A|^{1 + \delta_{1}/d}.
 \end{equation}
\end{theorem}

We now remark upon the sum-product conjecture, in the setting of $d \times d$ matrices. Unlike the case of real numbers, whenever $d \geq 2$, we have arbitrarily large finite sets of $d \times d$ matrices with real entries, that have a small sumset and product set. In particular, we consider the following example from \cite{SW2018}, namely
\begin{equation} \label{defdn}
 D_N = \bigg\{ \begin{pmatrix}  1 & i/N \\ 0 & 1 \end{pmatrix} \ \bigg| \ 1 \leq i \leq N  \bigg\}. 
 \end{equation}
%We see that for all distinct $a, b \in D_N$, we have $\det(a-b) = 0$, as well as 
We see that
\[ |D_N + D_N| = |D_N \cdot D_N| = 2N - 1 \ll |D_N|.\]
Thus, in the case of general $d \times d$ matrices with real entries, the sum-product conjecture fails to be true by a large margin. Moreover, even if we specify $A$ to be a set of $d \times d$ diagonal matrices, we see that example $\eqref{limit}$ restricts the range of $\delta$ for which we can expect a variant of $\eqref{sumprodreal}$ to hold. Hence, a more interesting question in this setting is finding what conditions must a set $A$ of $d \times d$ matrices with real entries satisfy, such that we have
\[ |A+A| + |A \cdot A| \gg |A|^{1+ \delta} , \]
for some fixed $\delta > 0$. A further quantitative aspect of this question is studying the relation between the conditions assumed on the set $A$, and the range of valid $\delta$ that these conditions prescribe.
\par
% Moreover, understanding the relation between the choice of conditions and the resulting value of $\delta$ is also of much interest. 

This problem has been analysed for several different cases, with varying quantitative bounds. For instance, Chang \cite{Ch2007} considered two cases, first when $A$ is a set of symmetric matrices with real entries, and second, when all elements in the set $(A-A)\setminus \{0\}$ are invertible. The latter situation was further explored by Solymosi and Vu \cite{SV2009}, who showed that upon the additional assumption that the condition number\footnote{See definition $\eqref{rice}$.} of all elements of $A$ was uniformly bounded, one can obtain significantly strengthened estimates. For our purposes, we focus on the case when $A$ is a set of $d \times d$ diagonal matrices, but we do not impose any further restrictions on the invertibility of non-zero elements of the set $A-A$, nor do we make any assumptions on the boundedness of the condition number of elements of $A$. While this setting is a special case of $A$ being a set of symmetric matrices, the bounds that we obtain are quantitatively much more stronger than those in \cite{Ch2007}. We present these comparisons in a more detailed manner in \S2.
\par

We will now comment briefly on the sum-product phenomenon itself. The sum-product conjecture was first posed by Erd\H{o}s and Szemer\'{e}di in \cite{ES1983}. Since then, numerous authors have worked on estimates of the form $\eqref{sumprodreal}$ in the case of of $A$ being either a subset of real numbers or some finite field. We refer the reader to \cite{GS2016} and \cite{Sh2018} for more details on these results. In recent years, considerable work has also been done to extend these bounds to other rings and fields. For instance, we now have sum-product estimates for complex numbers (see \cite{Ch2003, KR2013, So2005a}), quaternions (see \cite{BL2019, Ch2003, SW2018}), square matrices (see  \cite{Ch2007, ST2012, SV2009, SW2018, Ta2009}) and Function fields (see \cite{BJ2014}). 
\par

We will use \S3 to prove Theorem $\ref{main}$, and so, we end this section with a brief outline of our proof. The core of our argument rests on analysing how our set $A$ interacts with a particular class of affine subspaces, which we will call \emph{axis aligned affine subspaces}. Our proof then splits into two cases. In the first case, we will assume that a significant fraction of $A$ lies in a collection of axis aligned affine subspaces. We will use this to extract a rough Cartesian-product like structure in our set $A$, which will then allow us to perform induction on the dimension $d$. In the second case, we will use the absence of this Cartesian-product type structure in $A$ to directly reduce our problem to the one-dimensional case. 
\par

{\bf Acknowledgements.} The author's work was supported in part by a studentship sponsored by a European Research Council Advanced Grant under the European Union's Horizon 2020 research and innovation programme via grant agreement No.~695223. The author is grateful for support and hospitality from the University of Bristol and Purdue University. The author is thankful to Alberto Espuny D\'{i}az for pointing to this problem, and to Trevor Wooley for helpful discussions. The author would also like to thank the anonymous referee for many helpful comments.

%---------------------------------------------------------------------------------------------------------------------------
%---------------------------------------------------------------------------------------------------------------------------
%---------------------------------------------------------------------------------------------------------------------------
%---------------------------------------------------------------------------------------------------------------------------
%---------------------------------------------------------------------------------------------------------------------------
%---------------------------------------------------------------------------------------------------------------------------

\section{Previous results}

We utilise this section to describe some of the earlier work on sum-product estimates for square matrices. We begin by recording a result of Chang \cite[Theorem B]{Ch2007} which states that for all $d \in \mathbb{N}$, there exists $\varepsilon_{d} >0$ such that for all finite, non-empty sets $A$ of $d \times d$ symmetric matrices with real entries, one has
\begin{equation} \label{chang07}
 |A+A| + |A\cdot A| \gg |A|^{1 + \varepsilon_{d}}. 
 \end{equation} 
In the case when we specify $A$ to be a collection of diagonal matrices, Theorem $\ref{sumprodd}$ quantitatively strengthens $\eqref{chang07}$, allowing $\varepsilon_{d} = \delta_{1}/d - o(1)$. Moreover, noting $\eqref{limit}$, we observe that this estimate is conjecturally optimal in the case $d=2$. 
\par
%
%In the case of diagonal matrices, Theorem $\ref{sumprodd}$ quantitatively strengthens a previous result of Chang \cite[Theorem B]{Ch2007}. In particular, Chang showed that for all $d \in \mathbb{N}$, there exists $\varepsilon_{d} >0$ such that if $A$ is a set of $d \times d$ symmetric matrices with real entries, then
%\begin{equation} \label{chang07}
% |A+A| + |A\cdot A| \gg |A|^{1 + \varepsilon_{d}}. 
% \end{equation}
%We see that our result allows $\varepsilon_{d} = \delta_{1}/d - o(1)$ when $A$ is a collection of $d \times d$ diagonal matrices. Moreover, noting $\eqref{limit}$, we observe that our estimate is conjecturally optimal in the case $d=2$. 
%\par

Before proceeding further, we note a preliminary definition. Thus, we write the condition number $\kappa(a)$ of a $d \times d$ matrix $a$ to be 
\begin{equation} \label{rice}
\kappa(a) = \sigma_{\max}(a)\sigma_{\min}(a)^{-1},
\end{equation}
where $\sigma_{\max}(a)$ and $\sigma_{\min}(a)$ are the largest and smallest singular values of $a$ respectively. In particular, the smaller the condition number of a matrix, the farther it is from being singular. 
\par

As we mentioned previously, work has been done on sum-product estimates for more general square matrices than just diagonal or symmetric matrices, but Theorem \ref{sumprodd} and inequality $\eqref{chang07}$ seem to be the only results that do not assume that
\begin{equation} \label{condition1}
\det(a-a') \neq 0 \ \text{for all distinct} \ a,a' \in A,
\end{equation}
or require that there exists some small $\kappa > 0$ such that
\begin{equation} \label{condition2}
 \kappa(a) \leq \kappa \ \text{for all} \ a \in A,
 \end{equation}
where $\kappa(a)$ is the condition number of $a$. Consequently, the techniques that were used in previous results do not seem to generalise directly in our setting. 
\par

We will now make some remarks regarding conditions $\eqref{condition1}$ and $\eqref{condition2}$. If we assume our set $A$ in Theorem $\ref{sumprodd}$ satisfies $\eqref{condition1}$, then we directly obtain
\[ |A+A| + |A\cdot A| \gs |A|^{1 + \delta_1},\]
in which case we can use sum-product estimates for real numbers to get much stronger lower bounds. Thus for diagonal matrices, the more difficult case is when $A$ does not satisfy condition $\eqref{condition1}$. On the other hand, for more general matrices, $\eqref{condition1}$ seems to be a necessary condition. In particular, we consider the set $D_N$ as defined in $\eqref{defdn}$. We see that for all distinct $a, b \in D_N$, we have $\det(a-b) = 0$, as well as 
\[ |D_N + D_N| = |D_N \cdot D_N| = 2N - 1 \ll |D_N|.\]
\par

As for the second condition, there are two results that we will mention. In order to state these, we first give a preliminary definition. Given $\kappa \geq 1$ and a finite, non-empty set $A$ of $d \times d$ matrices with complex entries, we write $A$ to be $\kappa$-\emph{well conditioned} if $A$ and $\kappa$ satisfy $\eqref{condition2}$. Let $A$ be a finite, non-empty, $\kappa$-well conditioned set of $d \times d$ matrices $A$ with complex entries.  Solymosi and Vu \cite{SV2009} showed that if $\eqref{condition1}$ holds, one has
 \begin{equation} \label{SoVu}
  |A+A| + |A\cdot A| \gg_{\kappa,d} |A|^{1 + 1/4}. 
  \end{equation}
Similarly, Solymosi and Wong \cite{SW2018} proved that if all the elements of $A$ are invertible, and if for all $a,b,c,d \in A$, one either has
\[ a\cdot b^{-1} = c \cdot d^{-1} \ \text{or} \ \det(a \cdot b^{-1} - c \cdot d^{-1}) \neq 0,\]
then one can show that
\begin{equation} \label{SoWo}
 |A+A| + |A\cdot A| \gg_{\kappa,d} |A|^{1 + 1/3}{(\log |A|)^{-1/3}}. 
 \end{equation}
\par

We note that while $\eqref{SoVu}$ and $\eqref{SoWo}$ imply better and uniform exponents in $\eqref{sumproddiag}$, they only work for $\kappa$-well conditioned sets. Moreover, as the implicit constants in $\eqref{SoVu}$ and $\eqref{SoWo}$ depend on $\kappa$,  the parameter  $\kappa$ can not grow too fast. In particular, if our set $A$ in Theorem $\ref{sumprodd}$ was $\kappa$-well conditioned, we would have
\[ \kappa^{-1} \leq |a_{ii}|/|a_{jj}| \leq \kappa \ (1 \leq i,j \leq d),\]
for each $a = \{a_{ij} \}_{1 \leq i,j \leq n} \in A$. This restricts how sparse our set $A$ can get. Thus while Theorem $\ref{sumprodd}$ does not improve the known results for well conditioned sets of diagonal matrices, it is applicable to a more general class of diagonal matrices. 
\par

%---------------------------------------------------------------------------------------------------------------------------
%---------------------------------------------------------------------------------------------------------------------------
%---------------------------------------------------------------------------------------------------------------------------
%---------------------------------------------------------------------------------------------------------------------------
%---------------------------------------------------------------------------------------------------------------------------
%---------------------------------------------------------------------------------------------------------------------------

\section{Proof of Theorem $\ref{main}$}

We use this section for proving Theorem $\ref{main}$. Our proof will proceed through induction on the dimension $d$. Our base case will be when $d=1$, which follows from the definition of $\delta_1$. Thus we mainly need to focus on the inductive step. Consequently, we can assume that Theorem $\ref{main}$ holds for all finite, non-empty sets of $\mathbb{R}^{d'}$ where $1 \leq d' < d$. With this in mind, we define the parameter $\delta_{u} = \delta_1 u^{-1}$ for each $u \in \mathbb{N}$. 
\par

Let $H$ be an affine subspace of $\mathbb{R}^d$. We write $H$ to be an axis aligned affine subspace if $H = X_1 \times X_2 \times \dots \times X_d$, where $X_i = \{a_i\}$ for some $a_i \in \mathbb{R}$, or $X_i = \mathbb{R}$, for each $1 \leq i \leq d$.

\begin{lemma} \label{aasub}
Let $H$ be an axis aligned affine subspace of dimension $1 \leq r \leq d-1$ in $\mathbb{R}^d$, and let $A$ be a finite, non-empty subset of $H$. Then we have
\[|A+A| + |A\cdot A| \gs_{r} |A|^{1 + \delta_r}. \]
\end{lemma}

\begin{proof}
Without loss of generality, we can assume that 
\[H = \{a_1\} \times \{a_2\} \times \dots \times \{a_{d-r}\} \times \mathbb{R} \times \dots \times \mathbb{R}. \]
We define a map $\pi : \mathbb{R}^d \to \mathbb{R}^r$ such that $\pi(x_1, \dots, x_d) = (x_{d-r+1}, \dots, x_d)$. We write $A' = \pi(A) \subseteq \mathbb{R}^r$, and from our induction hypothesis, we have
\begin{equation} \label{triv111}
 |A'+A'| + |A'\cdot A'| \gs_{r} |A'|^{1 + \delta_1 /r} \gs_{r} |A'|^{1 + \delta_r} .
 \end{equation}
For each $a' \in A'$, we fix a corresponding element $\pi^{-1}(a') \in A$ such that $\pi(\pi^{-1}(a')) = a'$. We note that for each sum $a' + b'$ in $A' + A'$, there is at least one corresponding sum $\pi^{-1}(a') + \pi^{-1}(b')$ in $A+A$. Moreover, if $a'+b' \neq c'+d'$, then $\pi^{-1}(a') + \pi^{-1}(b') \neq \pi^{-1}(c') + \pi^{-1}(d')$. Thus we have 
\[ |A+A| \geq |A' + A'|. \]
A similar argument for product sets shows that
\[ |A \cdot A| \geq |A' \cdot A'| . \]
We combine these two inequalities with $\eqref{triv111}$ to prove the lemma.
\end{proof}

We now begin the inductive step. Let $A \subseteq \mathbb{R}^d$ be a finite, non-empty set. Using the pigeonhole principle, we find a set $A_1 \subseteq A$ such that $|A_1| \geq 3^{-d}|A| $, and for every choice of $a , b \in A_1$, we have either
\begin{equation} \label{fml1}
 a_i b_i > 0 \ \text{or} \ a_i = b_i = 0, \ \text{for each} \ 1 \leq i \leq d.
 \end{equation}
Thus if we prove Theorem $\ref{main}$ for the set $A_1$, we can use the fact that 
\[ |A+A| + |A\cdot A| \geq |A_1 + A_1| + |A_1 \cdot A_1|, \ \text{and} \ |A_1| \gg_{d} |A|  \]
to finish the proof. This means that from this point, we can assume that our set $A$ satisfies $\eqref{fml1}$.
\par

Let $M >1$ be a large enough constant depending only on $d$, and let $\mathscr{F}$ be the collection of axis aligned affine subspaces $H$ that contain at least one element of $A$. We write $\mathscr{F}_1$ to be a subset of $\mathscr{F}$ such that for all $H \in \mathscr{F}_1$, we have $|H \cap A| \geq M$.
\par

We first consider the case when
\begin{equation} \label{structurecase}
\sum_{H \in \mathscr{F}_1} | H \cap A| \geq |A|/10^d.
\end{equation}
Note that up to translation, we have at most $2^d$ types of axis aligned affine subspaces in $\mathbb{R}^d$.

Thus, we can apply the pigeonhole principle along with $\eqref{structurecase}$, to find $\mathscr{F}_2 \subseteq \mathscr{F}_1$ such that 
\begin{equation} \label{pgn1}
\sum_{H \in \mathscr{F}_2} | H \cap A| \geq |A|/20^d,
\end{equation}
and all the affine subspaces in $\mathscr{F}_2$ are translates of some $r$-dimensional subspace, with $1 \leq r \leq d-1$. In other words, all affine subspaces in $\mathscr{F}_2$ are parallel, and consequently, disjoint. We see that $\eqref{pgn1}$ implies that
\[ \sum_{j=0}^{\ceil{\log|A|}} \sum_{\substack{  2^{j} \leq | H \cap A|< 2^{j+1} \\ H \in \mathscr{F}_2 }} |H \cap A| \geq |A|/20^d. \]
We now use pigeonhole principle to infer that there exists $\mathscr{F}_3 \subseteq \mathscr{F}_2$ and $I \in \mathbb{N}$ such that $M \leq 2^I \leq |A|$, and for each affine subspace $H \in \mathscr{F}_3$, we have 
\[ 2^{I} \leq |H \cap A| < 2^{I+1}, \] 
and
\[ \sum_{H \in \mathscr{F}_3 } |H \cap A| \gg_{d} |A|/{\log|A|}.  \]
This implies that
\begin{equation} \label{pgn2}
|\mathscr{F}_3| 2^{I} \gg_{d} |A|/{\log|A|}.
\end{equation}
\par

We now prove the following proposition.
\begin{Proposition} \label{induction1}
We have
\begin{equation*} 
|A+A| + |A\cdot A| \gs_{d} |A|( 2^{I \delta_{r}} + |\mathscr{F}_3|^{\delta_{d-r}}  ) .
\end{equation*}
\end{Proposition}

\begin{proof}
For simplicity, we will write $\mathscr{F}_3 = \{ H_1, \dots, H_{m} \}$ for some integer $m = |\mathscr{F}_3| \leq |A|$, and for each $1 \leq i \leq m$, we will write $B_i = A \cap H_i$. By definition of $\mathscr{F}_3$, we see that $|B_i| \geq 2^{I}$ for each $1 \leq i \leq m$. We note that for each $i \neq j$, the sets $B_i + B_i$ and $B_j + B_j$ are disjoint. To see this, we first remark that each $H_i$ is a translate of the same axis aligned subspace. Thus, without loss of generality, we have
\begin{equation} \label{defineai}
 H_i = \{a_i\} \times \mathbb{R}^{r} \ \text{with} \ a_i \in \mathbb{R}^{d-r}, \ \text{for each} \ 1 \leq i \leq m. 
 \end{equation}
This implies that $H_i + H_i = \{a_i + a_i \} \times \mathbb{R}^{r}$, and thus, $H_i + H_i$ is disjoint from $H_j + H_j$ whenever $i \neq j$. As $B_i + B_i \subseteq H_i + H_i$, our claim is proven. 
\par

Using condition $\eqref{fml1}$, we can argue similarly for the sequence of sets $B_i \cdot B_i$ for $1 \leq i \leq m$. Moreover, as each $B_i$ is contained in an axis aligned affine subspace $H_i$ of dimension $r$, we use Lemma $\ref{aasub}$ to get
\[ |B_i + B_i| + |B_i \cdot B_i| \gs_{r} |B_i|^{1+ \delta_r}. \] 
Thus we have
\begin{align*} 
|A+A| + |A \cdot A| & \geq \sum_{i=1}^{m}|B_i + B_i| + |B_i \cdot B_i|  \\
& \gs_{d} m|B_i|^{1+ \delta_r} \geq |\mathscr{F}_3| 2^{I} 2^{I \delta_r}.
\end{align*}
Combining $\eqref{pgn2}$ with this, we get
\begin{equation} \label{ineqone}
|A+A| + |A \cdot A| \gs_{d} |A| 2^{I \delta_r} .
\end{equation} 
This proves one of the lower bounds in Proposition $\ref{induction1}$.
\par

We now show the second part of our lower bound. We begin by considering the set $B' = \{ a_1, a_2, \dots, a_m \}$, where $a_i$ is defined in $\eqref{defineai}$ for each $1 \leq i \leq m$. As $B' \subseteq \mathbb{R}^{d-r}$, the inductive hypothesis implies that
\begin{equation} \label{proppt1}
 |B'+B'| + |B' \cdot B'| \gs_{d-r} |B'|^{1+ \delta_{d-r}} .
 \end{equation}
Moreover, given $a_i, a_j, a_k, a_l \in B'$, if 
\[ a_i + a_j \neq a_k + a_l, \]
then $H_i + H_j$ and $H_k + H_l$ are disjoint, and consequently, $B_i + B_j$ and $B_k + B_l$ are disjoint. Furthermore, we have
\[ |B_i + B_j| \geq |B_i| \geq 2^{I}, \ \text{for each} \ 1 \leq i, j \leq m. \]
Thus, we get 
\[ |A+A| \geq 2^{I} |B' + B'|. \]
We can similarly argue for the case of product sets to get
\[ |A \cdot A| \geq 2^{I} |B' \cdot B'|. \]
Combining these with $\eqref{proppt1}$, we see that
\[ |A+A| + |A \cdot A| \geq 2^{I}( |B'+B'| + |B' \cdot B'|) \gs_{d-r} 2^I |B'|^{1+ \delta_{d-r}}.\]
Since $|B'| = |\mathscr{F}_3|$, we apply $\eqref{pgn2}$ to get
\[  |A+A| + |A\cdot A| \gs_{d} 2^{I}|\mathscr{F}_3|^{1+ \delta_{d-r}} \gs_{d} |A||\mathscr{F}_3|^{\delta_{d-r}}. \]
This, along with $\eqref{ineqone}$, proves Proposition $\ref{induction1}$.
\end{proof}

We now combine $\eqref{pgn2}$ with Proposition $\ref{induction1}$ to get
\[ |A+A| + |A\cdot A| \gs_{d} |A|  (2^{I \delta_{r}} + |A|^{\delta_{d-r}} 2^{-I\delta_{d-r}} ).\]
Using elementary optimisation, we note that  
\[ x^{\delta_{r}} + |A|^{\delta_{d-r}} x^{-\delta_{d-r}} \geq  |A|^{1/(\delta_{r}^{-1} + \delta_{d-r}^{-1})} \]
for all $x$ in the domain $[1, 2|A|]$. Consequently, we have
\begin{equation} \label{optim}
2^{I \delta_{r}} + |A|^{\delta_{d-r}} 2^{-I\delta_{d-r}} \geq |A|^{1/(\delta_{r}^{-1} + \delta_{d-r}^{-1})}
\end{equation}
for each choice of $I \in [0, \ceil{\log|A|}]$. Since $\delta_{r} = \delta_{1}/r $ and $\delta_{d-r} = \delta_1/(d-r)$, we get
\[ |A|^{1/(\delta_{r}^{-1} + \delta_{d-r}^{-1})} = |A|^{\delta_1 d^{-1}}. \]
This, in turn, implies that 
\[ |A+A| + |A \cdot A| \gs_{d} |A| |A|^{\delta_1 d^{-1}} \gs_{d} |A|^{1 + \delta_{d}}.\] 
Thus we are done when $\eqref{pgn1}$ holds.
\par

We now assume that $\eqref{pgn1}$ does not hold, that is, 
\[ \sum_{H \in \mathscr{F}_1} | H \cap A| < |A|/10^d. \]
This implies that if we consider the set $A' = A \setminus (\cup_{H \in \mathscr{F}_1} H)$, we have
\[ |A'| \gg_{d} |A|, \]
and each axis aligned affine subspace $H$ contains at most $M$ points of $A'$. Given any $a \in \mathbb{R}$, consider the axis aligned affine subspace $H_a = \{a\} \times \mathbb{R}^{d-1}$. For each $a \in \mathbb{R}$, we have $|H_a \cap A'| < M$. Thus if we consider $A_1 \subseteq \mathbb{R}$ to be the set
\[A_1 = \{  a \ | \ (a, a_2, \dots , a_d) \in A' \ \text{for some} \ (a_2, \dots, a_d) \in \mathbb{R}^{d-1}  \}, \]
we see that
\[ |A_1| \geq |A'|/M \gg_{d} |A|/M. \] 
Moreover, as $A_1 \subseteq \mathbb{R}$, we have
\[ |A_1 + A_1| + |A_1 \cdot A_1| \gs |A_1|^{1 + \delta_{1}} \gs_{d} |A|^{1+ \delta_{1}}M^{-1- \delta_{1}}. \]
%MORE EXPLANATION REQUIRED
This implies that
\[ |A+A| + |A \cdot A| \geq |A_1 + A_1| + |A_1 \cdot A_1| \gs_{d} |A|^{1+ \delta_{1}}M^{-1- \delta_{1}}. \] 
Choosing $M$ to be a large constant that depends on $d$, we get
\begin{equation} \label{whatdd}
 |A+A| + |A \cdot A| \gs_{d} |A|^{1+ \delta_{1}} \gs_{d} |A|^{1+ \delta_{d}}, 
 \end{equation}
in which case, we are done. This finishes the proof of Theorem $\ref{main}$.
\par
As a remark, we note that we can not further optimise our result just by choosing a larger value of $M$. For instance, we consider the case when $d=2$. In order to strengthen $\eqref{optim}$ in this case, $M$ needs to exceed the optimisation value $|A|^{1/2}$. But in this range, we have
\[ |A|^{1+ \delta_1} M^{-1 - \delta_1} < |A|^{ (1+ \delta_1) /2}, \]
which significantly weakens $\eqref{whatdd}$, and consequently, weakens our result overall.

\bibliographystyle{amsbracket}
\providecommand{\bysame}{\leavevmode\hbox to3em{\hrulefill}\thinspace}

\end{document}